\documentclass[a4paper]{article}
\usepackage[utf8]{inputenc}
\usepackage{amsmath}
\usepackage{amssymb}
\usepackage{amsfonts}
\usepackage{amsthm}
\usepackage{mathrsfs} 
\usepackage{xcolor}
\usepackage[margin=1.3in]{geometry}
\usepackage{hyperref}
\usepackage{xparse}
\usepackage{bm}
\usepackage{bbm}
\usepackage{tikz}
\usepackage{graphicx}
\usetikzlibrary{automata, positioning, arrows}

\usepackage{centernot}
\newcommand{\R}{\mathbb{R}}

\newcommand{\N}{\mathbb{N}}

\newcommand{\E}{\mathbb{E}}

\newcommand{\Prob}{\mathbb{P}}

\newcommand{\ep}{\epsilon}

\hypersetup{
    colorlinks=true,
    linkcolor=blue,
    filecolor=magenta,
    urlcolor=blue,
    citecolor=blue
}

\makeatletter
\newtheorem*{rep@theorem}{\rep@title}
\newcommand{\newreptheorem}[2]{%
\newenvironment{rep#1}[1]{%
 \def\rep@title{#2 \ref{##1}}%
 \begin{rep@theorem}}%
 {\end{rep@theorem}}}
\makeatother

\newtheorem{thm}{Theorem}
\newreptheorem{thm}{Theorem}

\newtheorem{result}{Result}[section]
\newtheorem{lem}[result]{Lemma}
\newtheorem{rmk}[result]{Remark}
\newtheorem{prp}[result]{Proposition}
\newtheorem{cor}[result]{Corollary}

\theoremstyle{definition}
\newtheorem*{defn}{Definition}

\newtheorem{prob}{Problem}

\theoremstyle{remark}

\DeclareMathOperator{\cost}{\bm{\mathsf{cost}}}
\DeclareMathOperator{\walk}{\bm{\mathsf{walk}}}
\renewcommand{\root}{\textrm{root}}%\bm{\mathsf{root}}}

\renewcommand{\c}[3][{}]{\cost_{#1}(#2,#3)}%\cost}
\renewcommand{\d}[3][{}]{\delta_{#1}(#2,#3)}
\newcommand{\w}[3][{}]{\walk_{#1}(#2,#3)}

\newcommand{\hide}[1]{}

\newcommand{\rough}[1]{}%\textbf{\textcolor{blue}{#1}}}
\definecolor{darkgreen}{RGB}{75,150,75}
\newcommand{\review}[1]{}%\textcolor{darkgreen}{#1}}

\newcommand{\zc}[1]{\textcolor{red}{zc: #1}}
\newcommand{\hides}[1]{}%1}
\newcommand{\pub}[1]{}%\textcolor{purple}{#1}}

\title{An asymptotically tight lower bound for superpatterns with small alphabets}
\author{Zach Hunter}
\date{\today}

\begin{document}
\tikzset{
->, % makes the edges directed
>=stealth, % makes the arrow heads bold
node distance=3cm, % specifies the minimum distance between two nodes. Change if necessary.
every state/.style={thick, fill=gray!10}, % sets the properties for each ’state’ node
initial text=$\textrm{root}$, % sets the text that appears on the start arrow
}

\maketitle
\begin{abstract}
    A permutation $\sigma \in S_n$ is a $k$-superpattern (or $k$-universal) if it contains each $\tau \in S_k$ as a pattern. This notion of ``superpatterns'' can be generalized to words on smaller alphabets, and several questions about superpatterns on small alphabets have recently been raised in the survey of Engen and Vatter. One of these questions concerned the length of the shortest $k$-superpattern on $[k+1]$. A construction by Miller gave an upper bound of $(k^2+k)/2$, which we show is optimal up to lower-order terms. This implies a weaker version of a conjecture by Eriksson, Eriksson, Linusson and Wastlund. Our results also refute a 40-year-old conjecture of Gupta.
\end{abstract}
\hide{Recent papers:

\url{https://arxiv.org/pdf/2004.02375.pdf}

\url{https://arxiv.org/pdf/1911.12878.pdf}}

\section{Introduction}

Given permutations $\tau \in S_k, \sigma \in S_n$, we say $\sigma$ \textit{contains} $\tau$ as a \textit{pattern} if there exist indices $1\le i_1< \dots<i_k\le n$ such that $\sigma(i_j) < \sigma(i_{j'})$ if and only if $\tau(j) < \tau(j')$ for all choices $j,j'$ (e.g., $312$ is contained in $2\underline{514}3$ as a pattern, we may choose $i_1,i_2,i_3 = 2,3,4$). We say that $\sigma\in S_n$ is a \textit{$k$-superpattern} if it contains each $\tau \in S_k$ as a pattern. Naturally, this leads us to consider the ``superpattern problem''.
\begin{prob}
For $k \ge 1$, let $f(k)$ be the minimum $n$ such that there exists $\sigma\in S_n$ which is a $k$-superpattern. What is the asymptotic growth of $f(k)$?
\end{prob}\noindent In 1999, Arratia \cite{arratia} showed that $(1/e^2-o(1))k^2 \le f(k) \le k^2$, hence $f(k)$ is well-defined.

There have been several competing conjectures about the asymptotic growth of $f(k)$. The conjecture relevant to this paper is that of Eriksson, Eriksson, Linusson and Wastlund, which claimed $f(k) = (1/2\pm o(1))k^2$ \cite{eriksson}. As some evidence towards this conjecture, Miller showed that there exist $k$-superpatterns of length $(k^2+k)/2$ (i.e., $f(k) \le (k^2+k)/2$) \cite{miller}. And later Engen and Vatter improved this to show $f(k) \le (k^2+1)/2$ \cite{engen}. However in forthcoming work \cite{hunter}, the author will show that $f(k)\le \frac{15}{32}k^2+O(k)$, refuting the claim that the constant $1/2$ is tight.

\hide{
In 2009, Miller showed that there exist $k$-superpatterns of length $(k^2+k)/2$ using the alphabet $[k+1]$ (i.e., $f(k;k+1) \le (k^2+k)/2$) \cite{miller}. The linear factor was improved in 2018 by Engen and Vatter, though this used an alphabet that grows quadratically with $k$ \cite{engen}. In forthcoming work \cite{hunter}, the author will show that $f(k)\le \frac{15}{32}k^2+O(k)$ using quadratically large alphabets. This contradicts a conjecture by Eriksson et al. \cite{eriksson} that $f(k) = (1/2\pm o(1))k^2$.}

In light of this, one is left to wonder if a revised version of the conjecture from \cite{eriksson} holds true. We answer this in the affirmative by considering a ``stricter regime'' of the superpattern problem which has received attention recently (see \cite{chroman,engen}).

\hide{
In this paper we prove a weakening of the conjecture by showing the constant $1/2$ is asymptotically tight in a stricter regime of the superpattern problem. Recently, the survey of Engen and Vatter \cite{engen} there have been breakthroughs in the superpattern problem and some of its variants \cite{chroman,he}. We study a stricter regime of the superpattern problem, which was mentioned in \cite[Section~6]{chroman}. Meanwhile, in this paper, we will show that the constant $1/2$ is asymptotically tight when we consider a stricter regime of the superpattern problem. In this paper, we 

Recently, there have been breakthroughs in the superpattern problem and some of its variants \cite{chroman,he}. In this paper, we study a certain regime of the superpattern problem, which was mentioned in \cite[Section~6]{chroman}. Our results are asymptotically tight in said regime. \zc{this sentence i can't read:} They have several implication of our results also refute a 40-year conjecture of Gupta \cite{gupta}; see Section~\ref{implications} for more details. \zc{usually, people introduce history before saying their own results. just something to think about.}}

The regime in question concerns ``alphabet size''. Instead of having $\sigma$ be a permutation, what if it was a word (i.e., sequence) on the alphabet $[r]:=\{1,\dots,r\}$? For $\sigma\in [r]^n$ and $\tau \in S_k$, we say $\sigma$ contains $\tau$ as a pattern for the same reasons as before (i.e., if there are indices $1\le i_1<\dots<i_k\le n$ such that $\sigma(i_j) < \sigma(i_{j'})$ if and only if $\tau(j)<\tau(j')$). As before, we say $\sigma \in [r]^n$ is a $k$-superpattern if it contains every $\tau \in S_k$ as a pattern. We define $f(k;r)$ to be the minimum $n$ such that there is a $\sigma\in [r]^n$ which is a $k$-superpattern.

One could revise this conjecture, by claiming in regimes with ``small'' alphabets, that the shortest $k$-superpatterns have a length of $(1/2\pm o(1))k^2$. In this paper, we prove the revised conjecture for the regime where $r=r_k = (1+o(1))k$. The lower bound is given by our main result.
\begin{thm}\label{opt}For every $\ep > 0$, there exists $\delta > 0$ so that the following holds for sufficiently large $k$. For\footnote{Throughout the paper, we omit floor functions when there is not risk for confusion.} $r_k = (1+\delta)k$ and $n < (1/2-\ep)k^2$, no word $\sigma \in [r_k]^n$ is a $k$-superpattern.
\end{thm}\noindent Hence, with Miller's construction (which uses the alphabet $[k+1]$, and thus shows $f(k;k+1)\le (k^2+k)/2$), we have asymptotically sharp bounds of the shortest superpatterns in this regime.
\begin{cor}\label{asy}Suppose $r_k = (1+o(1))k$ and also $r_k > k$ for all $k$. Then
\[f(k;r_k) = \left(\frac{1}{2}\pm o(1)\right)k^2.\]
\end{cor}

In Section~\ref{outline} we go over past lower bounds of $f$, and outline a proof of Theorem~\ref{opt}. In Section~\ref{notation} we go over notation. In Section~\ref{reduction} we go over a reduction which shows that Theorem~\ref{opt} follows from a more technical Theorem~\ref{goodwalks}, which we state later. 

We came across two proofs of Theorem~\ref{goodwalks}, we include both (but provide different levels of detail). In Section~\ref{det} we prove Theorem~\ref{goodwalks} by a simple coupling argument. In Section~\ref{alt} we sketch a second proof which uses the differential method. We believe our second proof is more likely to find applications in future research, however the first proof is more natural and was easier to present in full detail.

In Section~\ref{conclusion} we discuss some open problems and go over some results about the lower-order terms of our bounds. One part which may be of particular interest Section~\ref{implications}, where we refute a conjecture made by Gupta in 1981 \cite{gupta}, which was about the length of ``bi-directional circular superpatterns''.

\subsection{Past lower bounds and an outline of our proof}\label{outline}
We mention two trivial lower bounds for the length of superpatterns. Any $\sigma \in S_n$ contains at most $\binom{n}{k}$ permutations $\tau \in S_k$ as a pattern, since $\binom{n}{k}$ counts the number of choices of indices $1\le i_1<\dots<i_k\le n$. This implies $\binom{f(k)}{k}\ge k!$ must hold, which gives the bound $f(k) \ge (1/e^2-o(1))k^2$. Meanwhile, if $r_k = (1+o(1))k$, then one can get $f(k;r_k) \ge (1/e-o(1))k^2$ by a convexity argument (more specifically, one shows that any $\sigma \in [k]^n$ contains at most $(n/k)^k$ patterns of length $k$, and then one uses Remark~\ref{nfromk} which we mention shortly).

In 1976, Kleitman and Kwiatowski \cite{kleitman} used inductive methods to show that $f(k;k) \ge (1-o(1))k^2$ which is asymptotically tight (indeed, to see $f(k;k)\le k^2$ one can consider $1,\dots,k$ repeated $k$ times). But it was only in 2020 that Chroman, Kwan, and Singhal \cite{chroman} proved non-trivial lower bounds for superpatterns on alphabets larger than $[k]$. Basically, the methodology was based around ``encoding'' patterns in a more efficient manner. They show that typical choices of indices $1\le i_1<\dots<i_k\le n$ have many ``large gaps'' (choices $j$ where $i_{j+2}-i_j > Ck$ for a certain $C>0$), and that this property is particularly redundant (loosely, they create equivalence classes for choices of indices with many large gaps, and show that each equivalence class contain many choices of indices, yet few distinct patterns). This was used to show $f(k) \ge (1.000076/e^2)k^2$ for large $k$, and $f(k;(1+e^{-1000})k) \ge ((1+e^{-600})/e)k^2$ for large $k$.

In proving Theorem~\ref{opt}, we take a rather different approach than either of the previous papers which established non-trivial lower bounds (namely \cite{chroman,kleitman}). We actually reformulate the problem in terms of random walks on deterministic finite automata (DFAs). To get there, we need a definition and an observation.
\begin{defn}
For positive integers $k,n$, we let $F(k,n)$ be the maximum number of patterns $\tau \in S_k$ that a $\sigma \in [k]^n$ can contain.
\end{defn}
\begin{rmk}\label{nfromk}
 For any $\sigma \in [r]^n$, we have that $\sigma$ contains at most $\binom{r}{k}F(k,n)$ patterns $\tau \in S_k$. Consequently, if $r$ is such that\[\binom{r}{k}F(k,n) < k!\] then there is no $\sigma \in [r]^n$ which is a $k$-superpattern (i.e., $f(k;r)>n$).
\end{rmk}\noindent To confirm this remark, it suffices to verify the first sentence in Remark~\ref{nfromk}, which can be briefly justified as follows. Note that for each of the $\binom{r}{k}$ subsets $Y\subset [r]$ with $|Y|=k$, there are at most $F(k,n)$ permutations $\tau\in S_k$ which are a pattern of $\sigma|_{\sigma^{-1}(Y)}$. Conversely, if $\tau\in S_k$ is a pattern of $\sigma$, it is contained as a pattern of $\sigma|_{\sigma^{-1}(Y)}$ for some set $Y\subset [r]$ with $|Y| =k$.

Thus for fixed $\ep > 0$, we want to show that when $n < (1/2-\ep)k^2$ and $k$ is large that $F(k,n)$ will be ``extremely small''. We are able to show this by considering random walks on certain DFAs. What we specifically prove about DFAs is a bit technical, so we defer the rigorous statement to Section~\ref{thereduction}. Essentially, it implies a exponentially small upper bound for $F(k,n)/k!$ when $n < (1/2-\Omega(1))k^2$.
\begin{repthm}{goodwalks}[Informal statement]There exists a function $G(k,n,N)$ (which is defined in terms of a family of DFAs) such that 
\[F(k,(1/2-\ep)k^2) \le G(k,(1/2-\ep)k^2,k^2) .\]For fixed $\ep >0$, we will have \[G(k,(1/2-\ep)k^2,k^2) \quad \textrm{gets ``very small'' as }k\to \infty. \]
\end{repthm}\noindent Here, the notion of ``very small'' is such that Theorem~\ref{opt} will follow from an application of Remark~\ref{nfromk}.

Intuitively, one may expect our results to hold true by considering the following argument sketch. The rest of our paper will be dedicated to rigorously grounding this sketch.

Consider any $\sigma \in [k]^n$. Let $t$ be sampled from $[k]$ uniformly at random. For any $i_0 \in [n]$, we'll have that $\E[\inf\{i>i_0: \sigma(i) = t\}-i_0]\ge (k+1)/2$, which is minimized when $\sigma(i_0+1),\dots,\sigma(i_0+k)$ is a permutation (we use the convention $\infty-i_0 = \infty$ so that the quantity $\inf\{i>i_0: \sigma(i) = t\}-i_0$ is always well-defined). 

Thus, if $t_1,\dots,t_k$ are i.i.d. and sample $[k]$ uniformly at random, and we set $i_j = \inf\{i>i_{j-1}:\sigma(i) =t_j\}$ for each $j\in [k]$, then it should be exponentially likely (in terms of $k$) that $i_k-i_0 > (1/2-\ep)k^2$ (this essentially is due to a Chernoff bound). 

This quantity $i_k-i_0$ essentially tells us how long $\sigma$ needs to be so that we can ``embed'' $t_1,\dots,t_k$ into $\sigma$. In Section~\ref{reduction}, we go over our reduction from pattern containment to this deterministic embedding process, and then show how to use DFAs to track the quantity $i_k-i_0$. We conclude Section~\ref{reduction} by precisely stating Theorem~\ref{goodwalks} and showing how it implies Theorem~\ref{opt}. 

We then prove Theorem~\ref{goodwalks} in Section~\ref{det}. In our argument sketch above, we show that $i_k-i_0 < (1/2-\ep)k^2$ is exponentially unlikely when $t_1,\dots,t_k$ are sampled uniformly at random. What we do in Section~\ref{det} is show that the probability continues to be small when we condition on $t_1,\dots,t_k$ being a permutation. This is done by choosing $\alpha >0$ sufficiently small relative to $\ep$, and considering the behavior of substrings of length $\alpha k$. Here, if we choose the letters of our substring uniformly at random, the probability that our substring contains no repeated letters (i.e., it could be a substring of a permutation) is much larger than the probability that $i_{\alpha k+j}-i_j > (1/2-\ep)\alpha k^2$.

\subsection{Notation}\label{notation} 
For positive integers $n$ we let $[n]:= \{1,\dots,n\}$. We let $[\infty]:= \{1,2,3,\dots\} \cup\{\infty\}$. 

We use some standard asymptotic notation, detailed below. Let $f=f(k),g=g(k)$ be functions. We say $f = O(g)$ if there exists $C> 0$ such that $f\le Cg$ for sufficiently large $k$; conversely we say $f = \Omega(g)$ if there is $c> 0$ so that $f\ge cg$ for all large $k$. We use $o(1)$ to denote a non-negative\footnote{This is slightly non-standard, in most contexts $o(1)$ is allowed to be negative. We primarily use this convention to make the paper easier to read. We never implicitly make use of this convention in any of our proofs.} quantity that tends to zero as $k\to \infty$. Following \cite{keevash}, for a function $h= h(k)$, we say $h = f \pm g$ to mean $f-g\le h \le f+g$.

We remind the readers of the Kleene star operator. Given an alphabet (i.e., a set) $\Sigma$, we let $\Sigma^*$ denote the set of finite words on the alphabet $\Sigma$ (so $\Sigma= \bigcup_{n=0}^\infty \Sigma^i$).

For our purposes, a DFA is a $3$-tuple $D = (V,\delta,\root(D))$, where $V$ is the set (of ``states'') of $D$, $\delta:V\times \Sigma  \to V;(v,t)\mapsto \delta(v,t)$ is a transition function defined on some alphabet $\Sigma$, and $\root(D) \in V$ is the ``root'' of $D$. For the purposes of this paper, one may think of each DFA $D$ as being a rooted (not necessarily simple) directed graph, with its transition function, $\delta$, being a convenient way to describe walks on said graph.

Given a word $w \in [k]^*$ and $v\in V$ we define a walk in $D$, $\w{v}{w}$, as follows. Let $L$ be the number of letters in $w$, so $w = w_1,\dots,w_L$. We set $\w{v}{w} = v_0,\dots,v_L$, where $v_0 = v$ and for $j \in [L]$, $v_j = \d{v_{j-1}}{w_j}$.

Let $D$ be a DFA with a sets of states $V$, and suppose we have defined $\delta:V \times[k]\to V;(v,w)\mapsto \d{v}{w}$. We shall extend the function $\delta$ to the domain $V\times[k]^*$. Consider $w \in [k]^*$. If $w$ has length zero, then set $\d{v}{w} = v$. Otherwise, proceeding inductively, writing $w = w_1,\dots,w_L$, we can set $\d{v}{w} = \d{\d{v}{w_1}}{w_2,\dots,w_L}$.

\subsubsection{Cost}

Now we shall go over how we define a ``cost function''. We will start with an initial function $c:V\times [k]\to [\infty]$, and then extend it, similar to how we extended the transition function $\delta$. The end result will be a way to assign cost to walks that behaves additively; for those familiar with weighted graphs and the travelling salesman problem, we will effectively be translating the concept of weighted walks in terms of DFAs.

Let $D$ be a DFA with a sets of states $V$, and suppose we have defined $\cost:V \times[k]\to [\infty]$. We shall extend this to the domain $V\times[k]^*$. Given $v\in V$ and $w \in [k]^*$, we let $v_0,\dots,v_{|w|} = \w{v}{w}$, and set\[\c{v}{w} = \sum_{j \in [|w|]}\c{v_{j-1}}{w(j)}.\]In English, we initialize with net cost zero and do a walk according to $w$ that starts at state $v$ and let $\c{v}{w}$ be our net cost at the end of the walk. When doing the $j$-th step of our walk, we read the letter $w(j)$ while at state $v_{j-1}$ and shall increment our net cost by $\c{v_{j-1}}{w(j)}$ (if we think of $v_{j-1}$ as being a toll booth, this is the cost of taking the $w(j)$-th route of $v_{j-1}$).

A weighted DFA is simply a 2-tuple $(D,\cost)$ where $D$ is a DFA and $\cost$ is a cost function defined on $V$, the set of state of $D$. Given a weighted DFA $X =(D,\cost)$, we call $D$ the \textit{underlying DFA} of $X$. Also, for a weighted graph $X = (D,\cost)$, we will identify $X$ with $D$, so if we say something like ``let $V$ be the set of states of $X$'' we mean ``let $V$ be the set of states of $D$''. 

When talking about two DFAs $A,B$, we respectively denote the transition function of $A$ and the transition function of $B$ by $\delta_A$ and $\delta_B$. We similarly denote their walk functions by $\walk_A$ and $\walk_B$. In the same fashion, given two weighted DFAs $A,B$, each with their own cost function, we'll respectively denote them by $\cost_A$ and $\cost_B$. This allows us to compare functions when $A,B$ have a common set of states $V$. Thus, if we say $\c[A]{v}{t} \ge \c[B]{v}{t}$, this means that if we wanted to read the letter $t$ while at the state $v$, the associated cost of doing this in $A$ is at least as much as doing this in $B$.

We now introduce the concept of making a weighted DFA ``cheaper''. For a weighted DFA $X = (D,\cost)$ we say that $Y = (D',\cost')$ is a \textit{cheapening} of $X$ if $D = D'$ (i.e., they have the same underlying DFA) and for each $(v,t) \in V\times [k]$ we have that $\cost(v,t) \ge \cost'(v,t)$ (here $V$ is the set of states of $D$ and $[k]$ is the alphabet of letters which $D$ reads). 

The implication of this definition is that a cheapening will have a more relaxed cost function, that assign lower costs to all inputs (just like what would happen if one decreased the weights of some edges in an instance of the traveling salesman problem).
\begin{rmk}\label{cheap}If $B$ is a cheapening of $A$, then for $(v,w) \in V\times [k]^*$ we have that $\c[B]{v}{w}\le \c[A]{v}{w}$.
\begin{proof}
    Consider any $v \in V$ and $w\in [k]^*$. As $A,B$ have the same underlying DFA, we'll have that $\w[A]{v}{w} = v_0,\dots,v_{|w|} = \w[B]{v}{w}$. Hence, \[\c[A]{v}{w}- \c[B]{v}{w} = \sum_{j\in [|w|]} \c[A]{v_{j-1}}{w_j}-\c[B]{v_{j-1}}{w_j} \ge 0\](because $B$ is ``cheaper'' than $A$,\footnote{i.e., $\c[A]{v}{t} \ge \c[B]{v}{t}$ for all $(v,t) \in V\times [k]$} each summand is non-negative). It follows that $\c[A]{v}{w}\ge \c[B]{v}{w}$ as desired.
\end{proof}
\end{rmk}

Finally, here are some meta-notational conventions we will use. The symbol $\sigma$ will refer to a word we want to be a superpattern. The symbol $\tau$ will be an element of $S_k$, we'll wish to check if $\tau$ is a pattern of $\sigma$.

We use $i$ to denote an index of $\sigma$, $j$ to denote an index of $\tau$, $t$ to denote an image of $\tau$ (i.e., it would make sense to write ``with $\tau = t_1,\dots,t_k$'' or ``suppose $\tau(j) = t$'').

\section{Reduction}\label{reduction}

In this section, we will properly state Theorem~\ref{goodwalks} (which shall be proven in Section~\ref{det}), and prove that it implies Theorem~\ref{opt}. First, in Section~\ref{gstrat}, we formalize a ``greedy strategy'' for embedding $\tau$ into $\sigma$, and show that when $\sigma \in [k]^n$ and $\tau \in S_k$ that $\tau$ is a pattern of $\sigma$ if and only if the greedy strategy works. Then in Section~\ref{gdfa}, we will introduce a way to associate $\sigma \in [k]^n$ with a weighted DFA that will simulate this greedy embedding. 

Next in Section~\ref{kdfa}, we introduce a family of weighted DFAs, called $k$-DFAs, and show they generalize the weighted DFAs from Section~\ref{gdfa}. Lastly, in Section~\ref{thereduction} we first state Theorem~\ref{goodwalks} in terms of $k$-DFAs, and prove Theorem~\ref{opt} assuming this result.

\subsection{Greedy Strategy}\label{gstrat}

Let $\sigma\in [k]^n$ and $\tau \in S_k$. Since $\sigma$ uses the alphabet $[k]$, and $\tau$ uses every element of that alphabet, we have that $\tau$ is a pattern of $\sigma$ if and only if there are indices $1\le i_1< \dots <i_k \le n$ with $\sigma(i_j) = \tau(j)$ for each $j\in[k]$. Now, if such a choice/embedding of indices exist, then so will the ``greedy embedding'' of $\tau$ where we take $i_1 = \min\{i:i \in \sigma^{-1}(\tau(1))\}$ and iteratively for $j \in [k]\setminus \{1\}$ take $i_j = \min\{i>i_{j-1}: i\in \sigma^{-1}(\tau(j))\}$.\footnote{Indeed, suppose $i'_1<\dots < i'_k$ is one such embedding. We claim that the $i_j$ defined according to the greedy embedding will exist for all $j \in [k]$. First, we have that $\sigma(i'_1) =\tau(1)$, thus $i_1$ exists and we'll have $i_1 \le i'_1$. Then inductively, for any $j \in [k-1]$, assuming $i_j$ exists and $i_j\le i'_j$, we see that as $\sigma(i'_{j+1}) =\tau(j+1)$ and $i'_{j+1}> i'_j \ge i_j \implies i_{j+1} \le i'_{j+1}$. Hence we can construct $i_j$ for all $j\in [k]$ as required.}

Conversely, if we can construct $i_1,\dots,i_k$ according to the greedy embedding, it is clear that we'll have $i_1\ge 1$ and $i_k \le n$, which will imply $\sigma$ contains $\tau$ as a pattern. Hence, $\tau$ being a pattern of $\sigma$ is equivalent to being able to greedily embed $\tau$ into $\sigma$.

\subsection{Greedy DFA}\label{gdfa}

Given $\sigma \in [k]^n$, we shall create a weighted DFA $A_\sigma$ on $n+1$ states such that for $\tau \in S_k$, $\tau$ can be greedily embedded into $\sigma$ if and only if $c_\tau\le n$, where $c_\tau$ is the ``cost'' of the walk which $\tau$ induces in $A_\sigma$. We start by letting the states of $A_\sigma$ be $V=\{0\}\cup [n]$, with $0$ being the root. We will now define the transition function $\delta$ and the associated cost function $\cost$ on the domain $V\times[k]$. See Figure~\ref{fig:gdfa} for an example.

For $v \in V$ and $t \in [k]$, we let $u=u(v,t) = \inf\{i \in \sigma^{-1}(t): i>v\}$. If $u < \infty$, then $u\in [k] \subset V$, thus we define $\d{v}{t} = u$ and $\c{v}{t} = u-v$. Otherwise, if $u = \infty$, we let $\d{v}{t} = v$ and $\c{v}{t} = \infty$. 

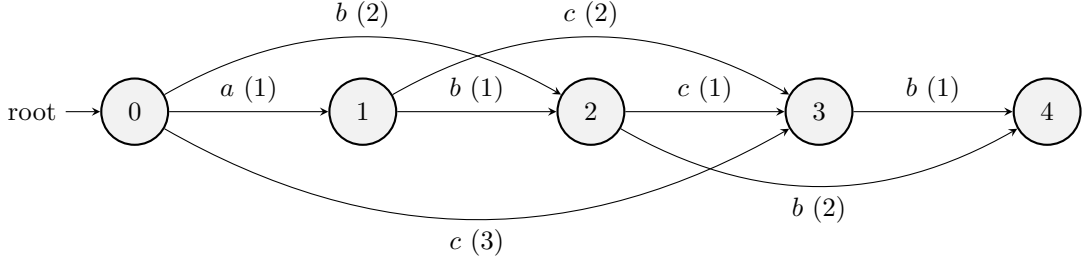
\begin{figure}[ht]
\centering\begin{tikzpicture}
\node[state, initial] (q0) {$0$};
\node[state, right of=q0] (q1) {$1$};
\node[state, right of=q1] (q2) {$2$};
\node[state, right of=q2] (q3) {$3$};
\node[state, right of=q3] (q4) {$4$};
\draw (q0) edge[above] node{$a$ (1)} (q1)
(q0) edge[bend left, above] node{$b$ (2)} (q2)
(q0) edge[bend right, below] node{$c$ (3)} (q3)
(q1) edge[above] node{$b$ (1)} (q2)
(q1) edge[bend left, above] node{$c$ (2)} (q3)
(q2) edge[bend right, below] node{$b$ (2)} (q4)
(q2) edge[above] node{$c$ (1)} (q3)
(q3) edge[above] node{$b$ (1)} (q4);
\end{tikzpicture}
\caption{A sketch of $A_\sigma$ where $\sigma = a,b,c,b$ (here we use the alphabet $\{a,b,c\}$ rather than $[3]$ for clarity). The labels of the edges are of the form ``$x$ (y)'' where $x \in \{a,b,c\}$ is the letter being read and y is the cost of the step. All omitted edges are self-loops with cost $\infty$.}
\label{fig:gdfa}
\end{figure}

As we went over in Section~\ref{notation}, we can extend $\delta,\cost$ to functions on the domain $V\times [k]^*$ by considering finite walks. We also now can define the walk function $\walk$ for $A_\sigma$. 

Now, given $v\in V, w \in [k]^*$ we consider $v_0,\dots,v_{|w|} = \w{v}{w}$. If $v_j = v_{j-1}$ for some $j \in [|w|]$, we say there was a failure. It is easy to see that if there is a failure, then $\c{v}{w} = \infty$, and otherwise we will have $\c{v}{w} = v_{|w|}-v_0$ (by induction).

We can now express pattern containment of permutations in terms of walks along $A_\sigma$. This is morally because $\w{0}{\tau}$ will mimic the greedy embedding of $\tau$, and has infinite cost if and only if the greedy embedding fails.
\begin{lem}For $\sigma \in [k]^n, \tau \in S_k$, we have that $\tau$ is a pattern of $\sigma$ if and only if $\c[A_\sigma]{0}{\tau}\le n$.
\begin{proof}Let $\cost,\walk$ be the cost and walk functions of $A_\sigma$. Consider any $w \in [k]^*$. We shall show that $w$ has a greedy embedding into $\sigma$ if and only if $\c{0}{w} \le n$. By Section~\ref{gstrat}, the result will follow, since $\tau$ will be a pattern of $\sigma$ if and only if it has a greedy embedding into $\sigma$.

By design/definition, we see that if $w\in [k]^*$ has a greedy embedding $i_1,\dots,i_{|w|}$ into $\sigma$, then $\w{0}{w} = 0,i_1,i_2,\dots,i_{|w|}$. Since $i_1 \ge 1 > 0$, and $i_j< i_{j+1}$ for $j \in [|w|-1]$, we get $\c{0}{w} = i_{|w|}\le n$ (because the walk does not have a failure). Meanwhile, we have that if $\c{0}{w} \le n$, then $0 = v_0 <v_1<\dots<v_{|w|}\le n$ with $v_0,\dots,v_{|w|}= \w{0}{w}$, making $v_1,\dots,v_{|w|}$ a greedy embedding of $w$ into $\sigma$.
\end{proof}
\end{lem}

\subsection{The family of \texorpdfstring{$k$}{k}-DFAs}\label{kdfa}

We will now define a family of weighted DFAs that will generalize the weighted DFAs $A_\sigma$ created in the last subsection. Let $D$ be a DFA with a set of states $V$ and a cost function $\cost:V\times[k]^*\to [\infty]$. We say $D$ is a \textit{$k$-DFA} if for each $v \in V$, we have that there is $\pi_v \in S_k$ such that $\pi_{v}(t) = \c{v}{t}$ for each $t\in [k]$.

Now we will show how $k$-DFAs ``generalizes'' the family of $A_\sigma$ from Section~\ref{gdfa}. Recall that given two weighted DFAs $X,Y$, we say $X$ is a cheapening of $Y$ if they both have the same underlying DFA, and we have $\c[X]{v}{t} \le \c[Y]{v}{t}$ for all $(v,t) \in V\times [k]$.
\begin{lem}\label{kcheap} Let $k$ be a positive integer. For any $\sigma \in [k]^n$, there exists a $k$-DFA $B_\sigma$ which is a cheapening of $A_\sigma$. 
\begin{proof} Let $V$ be the set of states for $A_\sigma$ and let $\cost$ be the cost function for $A_\sigma$ restricted to $V\times [k]$.

We shall take $B_\sigma$ to have the same underlying DFA as $A_\sigma$, and need to define some cost function $\cost_*$ for $B_\sigma$. It suffices to define $\c[*]{v}{t}$ for all $(v,t) \in V\times [k]$.

For each $v\in V$, we wish to find a permutation $\pi_v \in S_k$ such that $\pi_{v}(t)\le \c{v}{t}$ for all $t\in [k]$. We will then set $\c[*]{v}{t} = \pi_v(t)$ for all $(v,t)\in V\times[k]$.\hide{We may then set $\c[*]{v}{\cdot} = \pi_v$ for each $v\in V$.} If we can do this, then it is clear that $B_\sigma$ will be a $k$-DFA (by definition) and that it will be a cheapening of $A_\sigma$ (by our choices of $\pi_v$).

We now fix some $v\in V$, and find $\pi_v$. By construction of $A_\sigma$, we have that $\cost_v$ is injective on finite values. Indeed, for $t\in [k]$, we have $\c{v}{t} = c < \infty \implies \sigma(v+c) = t$, thus if $t,t'\in [k]$ have the same finite cost $c$ (starting at $v$) we have that $t = \sigma(v+c) = t'$.

Letting $T = \{t\in [k]: \c{v}{t} \le k\}$, we have that $\cost_v|_T$ is an injection into $[k]$ and $t\in [k]\setminus T $ will imply $\c{v}{t} > k$. Thus, it works to let $\pi_v = \pi \in S_k$ for any $\pi$ where $\pi|_T = \cost_v|_T$ (such $\pi$ will exist as $\cost_v|_T$ is an injection into $[k]$). 
\end{proof}
\end{lem}
Recalling Remark~\ref{cheap}, as $B_\sigma$ is a cheapening of $A_\sigma$, we have that $\c[B_\sigma]{v}{w}\le \c[A_\sigma]{v}{w}$ for all $(v,w) \in V\times [k]^*$. Hence, for any $\sigma \in [k]^n$, we get
\[ \{\tau \in S_k: \c[A_\sigma]{0}{\tau}\le n \} \subseteq \{\tau \in S_k: \c[B_\sigma]{0}{\tau}\le n \}\]
\hide{\begin{equation}\label{contain}
    \{\tau \in S_k: \c[A_\sigma]{0}{\tau}\le n \} \subseteq \{\tau \in S_k: \c[B_\sigma]{0}{\tau}\le n \} \tag{$\ast$}
  \end{equation}}where the set on the RHS is defined with respect to $B_\sigma$, which is a $k$-DFA.

\subsection{The Reduction}\label{thereduction}

We define $G(k,n,N)$ so that for any $k$-DFA $D$ on $N$ states, there are at most $G(k,n,N)$ ``permutational walks'' $w\in S_k$ where $\c{\root(D)}{w} \le n$. Observe that\[F(k,n) \le G(k,n,n+1)\le G(k,n,k^2)\]when $n\le k^2/2$ (here the first inequality follows by our previous work, and the second follows by the monotonicity of $G$ in the third variable). 

We can now make our original statement of Theorem~\ref*{goodwalks} precise.\begin{thm}[Formal statement]\label{goodwalks} Fix $\ep^*>0$. Then there exists $c>0$ such that for sufficiently large $k$,
\[G(k,(1/2-\ep^*)k^2,k^2) \le \exp(-ck)k!.\]\end{thm}
By Remark~\ref{nfromk}, we see Theorem~\ref{opt} will follow.
\begin{proof}[Proof of Theorem~\ref{opt} given Theorem~\ref{goodwalks}] Fix $\ep>0$.

We take $\ep^* = \ep$. We may apply Theorem~\ref{goodwalks} to get $c>0$ such that 
\[F(k,(1/2-\ep)k^2) \le \exp(-ck)k!\]for all sufficiently large $k$.

One can easily verify that there exists $\delta_0 >0$ such that  $2\delta  +\delta \log(\delta^{-1})\le c$ for all $\delta \in (0,\delta_0]$. We will take some $\delta =\min\{1,\delta_0\}$. Letting $r_k = (1+\delta)k$, a standard bound gives
\[\binom{r_k}{k} \le (e(1+\delta)\delta^{-1})^{\delta k} <\exp((2+\log(\delta^{-1}))\delta k) \le \exp(ck).\]Thus by Remark~\ref{nfromk}, we get that $f(k;r_k) >(1/2-\ep)k^2$ for sufficiently large $k$.
\end{proof}

\section{A coupling argument}\label{det}

\subsection{Machinery}\label{machinery}

In this subsection, we will fix some variables. We let $k$ be a (fixed) positive integer. We let $D$ be a (fixed) $k$-DFA with state set $V$; we respectively denote the transition, walk, and cost functions of $D$ by $\delta,\walk,$ and $\cost$.

We will say $w \in [k]^*$ is a \textit{permutational word} if $w(j) = w(j') \implies j=j'$ (i.e., if $w$ is injective). Note that permutational words will always use the alphabet $[k]$. Also, for $w = w_1,\dots,w_L$, and $E \subset [L]$, we write $w|_E$ to denote the word $w_{e_1},w_{e_2},\dots,w_{e_{|E|}}$, where $e_1< e_2<\dots<e_{|E|}$ are the elements of $E$ in increasing order.

We will make use of the following fact several times.
\begin{rmk}\label{uniform}
Suppose $w$ is sampled uniformly from permutational words of length $L$. For any $ E\subset [L]$, we have that $w|_E$ will sample permutational words of length $|E|$ uniformly at random. 
\end{rmk}\noindent This remark follows from basic properties of symmetry.

We will be concerned with bounding the following quantity, $P$.
\begin{defn}For $v \in V,\ep >0,L$, we define
\[P(v,L,\ep) = \Prob(\c{v}{w} < (1/2-\ep)kL)\]where $w$ is permutational word of length $L$ chosen uniformly at random.
\end{defn}
\noindent For convenience, for $v \in V, w\in [k]^*,\ep>0$, we say $w$ is \textit{$(v,\ep)$-bad} if $\c{v}{w} \le (1/2-\ep)k|w|$. Otherwise we say $w$ is $(v,\ep)$-good. Note that $P$ and this concept of ``goodness'' are defined with respect to $D$.

We now move on to proving some necessary lemmas.
\begin{lem}For any $v\in V,\ep> 0, L=L_0+L_1+\dots +L_M$\[P(v,L,\ep) \le P(v,L_0,\ep) + \sum_{u \in V} \sum_{m\in [M]}P(u,L_m,\ep).\]
\begin{proof} Set $I_0 = [L_0]$, and similarly for $m \in [M]$ set $I_m = [L_0+\dots+L_m]\setminus [L_0+\dots +L_{m-1}]$. Observe that $I_0,\dots,I_M$ partitions $[L]$. Also, for each $m \in \{0\}\cup[M]$ it is clear that $|I_m| = L_m$.

Consider a word $w\in [k]^L$ of length $L$. For each $m\in \{0\}\cup[M]$, let $w^m = w|_{I_m}$. Observe that for each $v \in V$, we can choose $u_1,\dots,u_M \in V$ so that
\[\c{v}{w} = \c{v}{w^0}+\sum_{m\in[M]}\c{u_m}{w^m}\] (indeed, we can start by taking $u_1 = \d{v}{w^0}$, and then for $m\in [M-1]$ take $u_{m+1} = \d{u_m}{w^m}$). This is because $w$ is the sequential concatenation of $w^0,w^1,\dots,w^M$.

Now suppose $w \in [k]^L$ is a $(v,\ep)$-bad word. It follows (essentially by pigeonhole) that there must exist some $m\in \{0\}\cup [M]$ where the event $E_m(w)$ is true, where
\begin{itemize}
    \item $E_0(w)$ is the event that $w^0$ is $(v,\ep)$-bad 
    \item and for $m\in [M]$, $E_m(w)$ is the event that $w^m$ is $(u_m,\ep)$-bad.
\end{itemize}

Let $w$ be sampled from permutational words of length $L$ uniformly at random. As above, for each $m\in \{0\}\cup[M]$ we define $w^m = w|_{I_m}$. Now, recalling Remark~\ref{uniform}, we will have that each $w^m$ will be sampled uniformly at random from permutational words of length $L_m$. 

Immediately, we see that the probability of the event $E_0(w)$ being true is exactly $P(v,L,\ep)$, by definition. We now consider each $m\in [M]$. As $w^m$ is a uniform random permutational word of length $L^m$, we'll get
\[\Prob(w^m \textrm{ is } (u,\ep)\textrm{-bad for some }u) \le \sum_{u \in V} P(u,L_m,\ep)\]
by union bound. Hence as the event $E_m(w)$ is contained in the event on the LHS, the probability of $E_m(w)$ occurring is upper-bounded by the RHS.

So by union bound we observe
\[\Prob(w\textrm{ is $(v,\ep)$-bad}) \le \sum_{m\in \{0\}\cup[ M]} \Prob(E_m(w)),\]which gives the desired result due to the bounds given in the preceding paragraph.
\end{proof}
\end{lem}
Writing $P(L,\ep):= \max_{v \in V}\{P(v,L,\ep)\}$, we immediately get
\begin{cor}\label{doubling} For, $\ep>0,L,M$\[P(ML,\ep) \le M|V| P(L,\ep).\]
\end{cor}

Next, we observe

\begin{lem}\label{forL} For $ \ep> 0, L$,
\[P(L,\ep) \le \frac{k^L(k-L)!}{k!}\exp(-\frac{\ep^2}{4} L).\]
\begin{proof}
    Let $w$ be uniform random word of length $L$. For each $v \in V$, we have that 
    \begin{align*}
        P(v,L,\ep) &= \Prob(w \textrm{ is $(v,\ep)$-bad}|w\textrm{ is permutational})\\
        &\le \frac{\Prob(w \textrm{ is $(v,\ep)$-bad})}{\Prob(w\textrm{ is permutational})}\\
    \end{align*}by Bayes' theorem.

    Immediately, we note that $\Prob(w\textrm{ is permutational}) = \frac{k!}{(k-L)! k^L}$, which justifies the first term in our lemma.
    
    Meanwhile, by a Chernoff bound \cite[Theorem~6.ii]{goemans}, we have that $\Prob(w \textrm{ is }(v,\ep)\textrm{-bad})\le \exp(-\frac{\ep^2}{4} L)$ as $\c{v}{w}$ is the sum of $L$ i.i.d. samples from the uniform distribution of $[k]$ (this is true by definition of $D$ being a $k$-DFA). This justifies the second term in our lemma.
    
    Hence, $P(v,L,\ep) \le \frac{k^L(k-L)!}{k!}\exp(-\frac{\ep^2}{4}L)$. As $v\in V$ was arbitrary, the same bound applies to $P(L,\ep)$, giving the result.
\end{proof}
\end{lem}
\subsection{Proof of Theorem~\ref{goodwalks}}

We require a standard bound for the birthday problem:
\begin{rmk}\label{birth}
There exists $\alpha_0>0$ such that for $\alpha \in (0,\alpha_0)$, if we take $L = \alpha k$, then we have that \[\frac{k^L(k-L)!}{k!} \le \exp((\alpha^2/2+\alpha^3/4)k).\]
\end{rmk}\noindent This follows from  \cite[Slide~11]{maji}.

We can now prove Theorem~\ref{goodwalks} by choosing $L$ appropriately.
\begin{proof}[Proof of Theorem~\ref{goodwalks}]
    Fix $\ep^* > 0$ and set $\ep = 2\ep^*/3$ so that \[(1/2-\ep)(1-\ep)>(1/2-\ep^*).\tag{$\dagger$}\label{epchoice}\]Without loss of generality, we may assume $\ep<\alpha_0$ where $\alpha_0$ is the constant from Remark~\ref{birth}.
    
    Let $D$ be any $k$-DFA with $k^2$ states. We define $P(\cdot,\cdot)$ and $(\cdot,\cdot)$-bad with respect to $D$ as we did in Section~\ref{machinery}. Now, we will take $L = \lfloor \alpha k\rfloor $ for some $\alpha\in (0,\ep)$ which we determine later. We shall bound $P(\ep,L)$ by directly applying Lemma~\ref{forL}.
    
    When $0<\alpha<\ep<\alpha_0$, the conclusion of Remark~\ref{birth} holds. Hence, plugging $L$ into Lemma~\ref{forL} gives
    \[P(L,\ep) \le \exp\left((\alpha^2/2+\alpha^3/4)k -\frac{\ep^2}{4}L\right) \le \exp\left((\alpha^2/2+\alpha^3/4-\frac{\ep^2}{4}\alpha)k+1\right) .\]
    (here the $+1$ is to account for $L$ being the floor of $\alpha k$) Taking $\alpha = \sqrt{\ep^2/2+1}-1 \in (0,\ep)$,\footnote{It should be clear that defining $\alpha$ in this way ensures $\alpha>0$; by checking derivatives one can confirm that $\ep> 0 \implies \alpha< \ep$. Hence $\alpha \in (0,\ep)$ as desired.} we get \[P(L,\ep)\le \exp(1-c_0k),\textrm{ with }c_0=\frac{\ep^2(\sqrt{\ep^2/2 +1}-1)}{8}.\]

    Next, we set $M = \lfloor k/L\rfloor$. Because $L \ge 1$, we will have $M \le k$, and by assumption $D$ has at most $k^2$ states. By Corollary~\ref{doubling},
    \[P(ML,\ep)\le k^3\exp(1-c_0k).\] 
    
    For later use, we remark that
    \[k-\ep k< k-\alpha k \le k-L < ML \tag{$\ddagger$}\label{floor}.\]The above follows from properties of the floor function and the fact that $\alpha < \ep$.
    
    Now, let $w \in S_k$ be sampled uniformly at random. By Remark~\ref{uniform}, $w':=w|_{[ML]}$ samples permutational words of length $ML$ uniformly at random. Trivially, \[\c{\root(D)}{w'}\le \c{\root(D)}{w}\]as $w'$ is a prefix of $w$. So, assuming $w'$ is $(\root(D),\ep)$-good, we get 
    \begin{align*}
        \c{\root(D)}{w} &\ge \c{\root(D)}{w'}\\
        &\ge (1/2-\ep)kML \\
        &>(1/2-\ep^*)k^2.\\
    \end{align*}
    (The last line quickly follows from \ref{epchoice} and \ref{floor}.) Thus, by our bound on $P(ML,\ep)$ from above 
    \[\Prob(\c{\root(D)}{w}\le (1/2-\ep^*)k^2) \le P(ML,\ep) \le k^3\exp(1-c_0k).\]As $D$ was arbitrary, this holds for all $k$-DFAs on $k^2$ states, thus 
    \[G(k,(1/2-\ep^*)k^2,k^2) \le k^3\exp(1-c_0k)k!.\]We conclude by fixing some choice of $c \in (0,c_0)$. By basic asymptotics, it follows that for sufficiently large $k$, we have
    \[G(k,(1/2-\ep^*)k^2,k^2)  \le \exp(-ck)k!.\]
\end{proof}

\section{An alternate approach}\label{alt}

In this section, we sketch another way to get bounds on $G$. Here, we break the cost of each walk into two parts, which we bound separately. 

Fix a $k$-DFA $D$. Suppose we sample $\tau \in S_k$ uniformly at random. We write $\tau = t_1,\dots,t_k$ and $v_0,\dots,v_k = \w[D]{\root(D)}{\tau}$. For each $j \in [k]$, let $C_j =\c[D]{v_{j-1}}{t_j}$. By definition of cost, \[\c[D]{\root(D)}{\tau} = \sum_{j=1}^k C_j \label{eq:cost}\tag{$*$}.\]

 Now, given $t_1,\dots, t_{j-1}$, there exists $S_j \subset [k],|S_j| = k-j+1$ such that $C_j$ samples $S_j$ uniformly at random (in particular, $t_1,\dots,t_{j-1}$ determines $v_{j-1}$ thus we get $S_j = \{\c[D]{v_{j-1}}{t}:t \in [k]\setminus \{t_1,\dots,t_{j-1}\}\}$). 

Let $X_j$ be such that $C_j$ is the $X_j$-th smallest element of $S_j$. Since $C_j$ samples $S_j$ uniformly, it follows that $X_j$ samples $[k-j+1]$ uniformly at random. We remark without proof that $X_1,\dots,X_k$ are independently distributed.

Next, we define $Y_j = C_j-X_j$, and observe that $Y_j$ is always non-negative. By \ref{eq:cost}, we get
\[\c[D]{\root(D)}{\tau} = \sum_{j=1}^k X_j +Y_j.\]
We shall now consider $\sum_{j=1}^k X_j$  and $\sum_{j=1}^k Y_j$ individually.

The first sum is not very complicated and does not depend on our choice of $D$. It suffices to apply Hoeffding's inequality.
\begin{lem}\label{Xbound} For any $\ep > 0$, and sufficiently large $k$,
\[\Prob(\sum_{j=1}^k X_j \le (1/4-\ep)k^2) < \exp( -32\ep^2 k/3).\]
\begin{proof} By linearity, \[\E\left[\sum_{j=1}^k X_j\right] = \sum_{j=1}^{k} \frac{k-j+1}{2} = \frac{1}{4}(k^2+k) > k^2/4.\]Meanwhile, for each $j$ the support of $X_j$ is contained in the interval $[1,k-j+1]$. We have that
\[\sum_{j=1}^k (k-j+1-1)^2 = \sum_{j=1}^{k-1} j^2 = \frac{1}{6}(k-1)k(2k-1) < k^3/3.\] 

Thus, applying a standard Hoeffding bound, we get 
\begin{align*}
    \Prob(\sum_{j=1}^k X_j \le (1/4-\ep)k^2) &< \exp\left(-\frac{2k^2 (4\ep k)^2}{k^3/3}\right)\\
    &= \exp(-32\ep^2 k/3).\\ 
\end{align*}
\end{proof}
\end{lem}

Next, we want to control the sum over $Y_j$. We first note
\begin{align*}
    Y_j &= \sum_{t=1}^{C_j} I(\c[D]{v_{j-1}}{t}=\c[D]{v_{j-1}}{t_{j'}}\textrm{ for some }j' \in [j-1])\\
    &\ge \min_{v \in V} \left\{\sum_{t=1}^{X_j} I(\c[D]{v}{t} =\c[D]{v}{t_{j'}} \textrm{ for some }j'\in [j-1] )\right\}.
\end{align*}
Thus, for $v \in V, j\in [k],x \in [k-j+1]$, we define \[T_{v,j,x} := \sum_{t=1}^x I(\c[D]{v}{t} = \c[D]{v}{t_{j'}} \textrm{ for some }j'\in [j-1])\]
\[\textrm{ and } T_{j,x} := \min_{v \in V}\{T_{v,j,x}\}.\]

We will next need two concentration results. These will allow us to bound $\sum_{j=1}^k Y_j$ in manner reminiscent to Riemann sums.
\begin{prp}\label{con1} Fix $\ep^* > 0$ and a positive integer $M$. There exists $c = c_{\ref{con1}}(\ep^*,M) > 0$ such that for each $m_1,m_2 \in [M-1]$, 
\[\Prob(|\{m_1k/M< j\le (m_1+1)k/M: X_j/(k-j+1)>\frac{m_2}{M}\}| < (1-\ep^*)\left(1-\frac{m_2}{M}\right)k/M)\le  \exp(-ck)\]when $k$ is sufficiently large.

We may in particular take $c_{\ref{con1}}(\ep^*,M) =\frac{1}{2}\left(\frac{\ep^*}{M}\right)^2$.
\end{prp}
\begin{prp}\label{con2} Fix $\ep^* > 0$ and a positive integer $M$. There exists $c = c_{\ref{con2}}(\ep^*,M) > 0$ such that for each $m_1,m_2 \in [M-1]$, 
\[\Prob(T_{j,m_2k/M} < (1-\ep^*)\left(\frac{m_2}{M}(j-1)\right) \textrm{ for some }\frac{m_1}{M}k<j\le \frac{m_1+1}{M}k)\le \exp(-ck)\]when $k$ is sufficiently large.

We may in particular take any $c_{\ref{con2}}(\ep^*,M) <\frac{1}{2}\left(\frac{\ep^*}{M}\right)^2$.
\end{prp}\noindent The first result immediately follows from a Chernoff bound, since the size of the set behaves exactly like a binomial random variable. To prove the second result it suffices to control $T_{j,v,m_2k/M}$ and then take a union bound over all $v,j$. To control $T_{j,v,m_2k/M}$, one can couple it with a binomial random variable $B$ with success probability slightly less than $m_2/M$ so that $\Prob(B> T_{j,v,m_2k/M})$ is exponentially small, and then apply a Chernoff bound. We leave the details as an exercise for the reader.

We note that Proposition~\ref{con2} is the only result whose proof will make use of the number of states in $D$ not being too large. In Section~\ref{bestdfa}, we give an example of $k$-DFA with $2^k$ states such that $\sum_{j=1}^k Y_j = 0$ always holds, thus limiting the growth of the number of states is necessary.

We now go over how to bound $\sum_{j=1}^kY_j$.
\begin{lem}\label{Ybound}Fix $\ep >0$. There exists $c > 0$ such that for sufficiently large $k$,
\[\Prob(\sum_{j=1}^kY_j < (1/4-\ep)k^2)<\exp(-ck).\]
\begin{proof}[Proof of Lemma~\ref{Ybound} given Proposition~\ref{con1} and Proposition~\ref{con2}]
Fix $\ep^*> 0$ and a positive integer $M$. Now assume the events of Proposition~\ref{con1} and Proposition~\ref{con2} for the given $\ep^*$ and $M$ do not hold for any $m_1,m_2 \in [M-1]$.

For $m_1 \in [M-1]$, let $E_{m_1} = [m_1k/M:(m_1+1)k/M]$. For $m_2 \in [M-1]$, let $F_{m_2} =\{ j: X_j/(k-j+1)>\frac{m_2}{M}\}$. We will have that
\begin{align*}
    \sum_{j\in E_{m_1}} Y_j &\ge \sum_{j \in E_{m_1}} T_{j,X_j}\\ 
    &\ge \sum_{m_2 \in [M-1]}\sum_{j \in E_{m_1}\cap F_{m_2}} (1-\ep^*)\frac{ (j-1)}{M}\\
    &\ge \frac{(1-\ep^*)}{M}\sum_{m_2 \in [M-1]}|E_{m_1}\cap F_{m_2}|k\frac{m_1}{M}\\
    &\ge \frac{(1-\ep^*)^2}{M^2}k^2 \sum_{m_2 \in [M-1]}(1-\frac{m_2}{M})\frac{m_1}{M}\\
\end{align*}(here the second inequality makes use of Proposition~\ref{con2} not holding and also applies telescoping; the last inequality makes use of Proposition~\ref{con1} not holding).

Hence,
\begin{align*}
    \sum_{j=1}^k Y_j &\ge \frac{(1-\ep^*)^2}{M^2}k^2\sum_{m_1 \in [M-1]} \sum_{m_2\in [M-1]} (1-\frac{m_2}{M})\frac{m_1}{M}\\
    &= \frac{(1-\ep^*)^2}{M^2}k^2\left(\frac{M-1}{2}\right)^2\\
    &\ge (1-\ep^*)^2(1-1/M)^2 \frac{1}{4}k^2\\
\end{align*}\noindent here the second line follows by separating the double sum into the product of two sums (which both happen to equal $(M-1)/2$).

Thus, if $\ep^*,M$ are such that $(1-\ep^*)^2(1-1/M)^2 \ge 1-4\ep$, the RHS will be at least $(1/4-\ep)k^2$.

Hence, the probability that $\sum_{j=1}^kY_j <(1/4-\ep)k^2$ is at most probability that there exists $m_1,m_2\in [M-1]$ such that the event from Proposition~\ref{con1} or Proposition~\ref{con2} holds with respect to the specified $\ep^*,M$. By union bound, this is at most \[(M-1)^2(\exp(-c_{\ref{con1}}(\ep^*,M)k)+\exp(-c_{\ref{con2}}(\ep^*,M)k)) \le \exp(-ck) \textrm{ for sufficiently large }k\] for any $c <\min\{c_{\ref{con1}}(\ep^*,M),c_{\ref{con2}}(\ep^*,M)\}$.
\end{proof}
\end{lem}
It is clear that combining Lemma~\ref{Xbound} and Lemma~\ref{Ybound} gives another proof of Theorem~\ref{goodwalks}.

\section{Conclusions}\label{conclusion}

\subsection{Lower order terms for \texorpdfstring{$f(k;k+1)$}{f(k;k+1)}}\label{loworder}

From Corollary~\ref{asy}, we know that $f(k;k+1) = (1/2\pm o(1))k^2$, meaning Miller's construction is optimal up to lower order terms. However, the statement of Theorem~\ref{opt} does not immediately yield any explicit function for this $o(1)$-term. We briefly mention an explicit function our methods yield. 

To prove $f(k;k+1) < n$, it suffices to show $kG(k,n,k^2) < k!$ (by Remark~\ref{nfromk}). The following comes from looking at the proof of Theorem~\ref{goodwalks}, and observing $c_0 > \ep^4/33$ for sufficiently small $\ep$ ($33$ may be replaced with any constant greater than $32$). \begin{rmk}For all sufficiently small $\ep>0$,
\[\ep^4 > \frac{33+132\log(k)}{k} \implies f(k;k+1)<(1/2-3\ep/2)k^2.\]
\end{rmk}\noindent Analyzing the work from Section~\ref{alt} should give a similar bound, where $33+ 4\log(k)$ is replaced by some other function of the same shape.

Thus, we can say
\begin{cor} For all $k$,
\[\frac{k^2}{2}-k^{7/4+o(1)}\le f(k;k+1)\le \frac{k^2+k}{2}.\]
\end{cor}\noindent It is interesting to note that the best lower bound of $f(k;k)$ is of the form $k^2-k^{7/4+o(1)}$ \cite{kleitman}. The lower bound for $f(k;k)$ was proved in 1976 and has remained unimproved for 45 years. It would be interesting to see if the lower-order error in the lower bound for $f(k;k)$ or $f(k;k+1)$ can be improved.

As we will demonstrate in Section~\ref{dstrat}, there is a limit to how well we can bound $f(k;k+1)$ by our methods. In particular, for large $k$ we have $G(k,k^2-k^{3/2},k+1) = \Omega(k!)$. In fact, a more careful calculation would give that $kG(k,k^2-h(k)k^{3/2},k+1) \ge  k!$ with $h(k)$ being some slowly growing function which is roughly $|\Phi^{-1}(C/\sqrt{k})|$ for a certain absolute constant $C>0$ (here $\Phi$ is the cdf of the standard normal distribution).

\subsection{Other Problems on \texorpdfstring{$k$}{k}-DFAs}\label{otherkdfa}We believe understanding the cost of permutational walks on $k$-DFAs might be of independent interest. We provide some useful constructions and ask a few future problems.

\subsubsection{Upper bound on \texorpdfstring{$G(k,n,N)$}{G(k,n,N)} independent of \texorpdfstring{$N$}{N}}\label{bestdfa}

We note that there's an ``optimally cheap'' $k$-DFA for reading permutations. By which we mean there is a $k$-DFA $A$ such that for any other $k$-DFA $B$, there exists a bijection $\phi:S_k \to S_k$ such that for $\tau \in S_k$ we have $\c[A]{\root(A)}{\pi} \le \c[B]{\root(B)}{\phi(\tau)}$. 

It follows that for any $k$-DFA $B$, that \[|\{\tau \in S_k:\c[B]{\root(B)}{\tau} \le n\}|\le |\{\tau \in S_k:\c[A]{\root(A)}{\tau} \le n\}|.\]Thus the RHS will exactly be $\max_{N}\{G(k,n,N)\}$.

We sketch on construction of $A$. For the set of states, $V$, we use all subsets of $[k]$ (with the empty set being the root). For $v \in V$, and $t \in [k]$, we set $\d{v}{t} = v\cup \{t\}$. For the cost, we impose for each $v\in V$, that $t \in v \iff \c{v}{t}> k-|v|$. Essentially, the DFA will remember which letters have been read thus far, and assigns the highest costs to these letters (since when reading a permutation, we never read a letter twice). 

To see optimality, it suffices to show that we'll always have $\sum_{j=1}^k Y_j = 0$ (here we use the terminology from Section~\ref{alt}). This follows immediately from how the cost is defined. If we've walked to a vertex $v$, then letters we've read while walking to $v$ is exactly the elements of $v$, and these will have greater cost at $v$ then any letter which is not an element of $v$ (and thus none of the summands $Y_j$ can be non-zero).

\subsubsection{\texorpdfstring{$k$}{k}-DFA's with many low cost permutations}\label{dstrat}

It would be interesting to better understand how fast $n_k$ must grow when
\[G(k,n_k,k^2) = \Omega(k!).\]Repeating the analysis from Section~\ref{loworder}, we get that $ n_k\ge k^2/2 - k^{7/4+o(1)}$ must hold.

We will describe a construction (provided by Zachary Chase in personal communication) of a $k$-DFA $D$ on $k+1$ states such that for ``many''  $\tau \in S_k$, $\c[D]{\root(D)}{\tau} \le k^2/2 - k^{3/2}$. This will show that its possible to have $n_k \le k^2/2 - \Omega(k^{3/2})$.

We first partition $[k]$ into two sets $A,B$ as evenly as possible, such that $|A| \le |B| \le |A|+1$. Out set of states will be $V:=\{-|A|,1-|A|,\dots,|B|\}$ with root $0$. 

For $t\in A$, we let $\d[D]{v}{t} = v-1$ if $v\neq -|A|$ and for $t\in B$ we let $\d[D]{v}{t} = v+1$ if $v\neq |B|$ (otherwise we let $\delta$ be constant, though this will not matter when reading permutations).

With $v_0,\dots,v_L = \w[D]{0}{w_1,\dots,w_L}$, we observe that we'll have $v_j = |B\cap \{w_1,\dots,w_j\}|-|A\cap \{w_1,\dots,w_j\}|$, unless there was some $j'<j$ where $w_{j'+1} = w_{j'}$. Whenever $w$ is a permutation, the second case will not happen, so $v_0,\dots,w_k := \w[D]{0}{w}$ satisfies \[v_j = |B\cap \{w_1,\dots,w_j\}|-|A\cap \{w_1,\dots,w_j\}|\] whenever $w\in S_k$.

For our cost function, we will assign the elements of $A$ lower weights when we are in a negative state and do the opposite otherwise. For simplicity, we consider the case where $k = 2m$, $A = [m],B = [2m]\setminus [m]$. Then for $v \in V,t\in [k]$, we let
\[\c[D]{v}{t}=\begin{cases}t &\textrm{if }v<0\\
                            t+m&\textrm{if $v\ge 0$ and }t\in A\\    t-m&\textrm{if $v\ge 0$ and }t\in B\\\end{cases}.\]

We now analyze the cost of reading permutations in $D$. We may write $\c[D]{v}{t} =mq(v,t)+r(t)$, where $q(v,t) \in \{0,1\}, r(t)\in [m]$ (it is easily verified that $r(t)$ does not depend on $v$). Thus, for $\tau \in S_k$, if $\w[D]{0}{\tau} = v_0,\dots,v_k$, then
\[\c[D]{0}{\tau} = \sum_{t \in [k]}r(t) + m\sum_{j\in [k]} q(v_{j-1},\tau(j)).\]
 Noting $\sum_{t \in [k]}r(t) =\frac{k^2}{4}-\frac{k}{2}\le k^2/4$, it remains to control the second term. 

Now, we claim (without proof) that if $\tau \in S_k$ is chosen uniformly at random, there is a coupling with $X_1,\dots,X_k$ (where $X_i$ are i.i.d. Bernoulli variables with $\Prob(X_i=1) = 1/2$) so that $X_j = 0 \implies q(v_{j-1},\tau(j)) = 0$. By Berry-Esseen Theorem, one can see that \[\Prob(\sum_{j=1}^k X_j \le k/2-2\sqrt{k})\to \Phi(-4)>0\] (where $\Phi$ is the cdf of the standard normal distribution). As $X_j \ge q(v_{j-1},\tau(j))$ for each $j$, it follows that for large $k$, \[\Prob\left(\sum_{j\in [k]}q(v_{j-1},\tau(j)) \le k/2-2\sqrt{k}\right) \ge \Phi(-4)/2\]
\[\implies G(k,k^2-k^{3/2},k+1) \ge \frac{\Phi(-4)}{2} k!.\]

\subsection{Refuting a conjecture of Gupta}\label{implications}

Lastly, we demonstrate how our result contradicts a conjecture by Gupta \cite{gupta} (see also the second item in the final section of \cite{engen}). This conjecture is concerned with ``bi-directional circular pattern containment''.

Essentially, given a word $w\in [r]^n$, we say $\tau \in S_k$ is a \textit{circular pattern} of $w$ if there exists $i \in [n]$ such that $\tau$ is a pattern of 
\[w(i),w(i+1),\dots,w(n),w(1),w(2),\dots,w(i-1).\]We say $\tau\in S_k$ is a \textit{bi-directional circular pattern} (BCP) of $w\in [r]^n$ if $\tau$ is circular pattern of  $w$ and/or $w's$ reversal, $w(n),w(n-1),\dots,w(2),w(1)$.

Gupta conjectured that for each $k$, there was $\sigma \in [k]^n$ with $n \le \frac{3}{8}k^2 +\frac{1}{2}$ such that each $\tau \in S_k$ is a BCP of $\sigma$. By definition of BCPs, this would mean that there exists $2n$ words $w_1,\dots,w_{2n}\in [k]^n$ such that for any $\tau \in S_k$, there exists $i\in [2n]$ such that $\tau$ is pattern of $w_i$. 

This would imply that $k! \le 2nF(k,n) \le k^2F(k,n)$. Hence, by our bounds on $F(k,n)$ we get a contradiction for large $k$. In fact, essentially repeating the analysis from Section~\ref{loworder}, we can show that if $\sigma\in [k]^n$ contains each $\tau \in S_k$ as a BCP, then $n \ge \frac{k^2}{2}-k^{7/4+o(1)}$. In 2012, Lecouturier and Zmiaikou proved that there exists $\sigma \in [k]^{k^2/2 +O(k)}$ which contain each $\tau \in S_k$ as a circular pattern (and hence as a BCP), thus our bound is tight up to lower-order terms \cite{lecouturier}.

\subsection{A 0-1 phenomenon}

In \cite[Section~6]{chroman}, it was asked how large must $n_k$ be for there to exist $\sigma \in [k]^{n_k}$ which contain almost all patterns in $S_k$ (i.e., what are the growth of sequences $n_k$ so that $F(k,n_k) = (1-o(1))k!$).  Again, the analysis of Section~\ref{loworder} shows that $n_k \ge k^2/2 - k^{7/4 + o(1)}$ is necessary for $F(k,n_k) = \Omega(k!)$ to hold. 

Meanwhile, if we consider the word $w_k^m$ obtained by concatenating $m$ copies of $1,2,\dots,k$, we have that $w$ contains all $\tau \in S_k$ with at least $k-m$ ascents (the number of ascents in a permutation $\tau \in S_k$ is the number of $j \in [k-1]$ such that $\tau(j)<\tau(j+1)$). By reversing the order of permutation $\tau \in S_k$ with $a$ ascents, you get a permutation with $k-a-1$ ascents. Thus, with $m = \lceil k/2\rceil$ we have that $w_k^{m}$ contains at least half of the $\tau \in S_k$ as a pattern (thus $n_k  = (k^2+k)/2$ satisfies $F(k,n_k) \ge k!/2$). 

Finally, using standard martingale concentration results (see e.g. \cite[Proposition~2.3]{alon}) if $m = k/2 +C\sqrt{k}$ then $w_k^m$ contains $(1-2\exp(-\Omega(C^2)))k!$ patterns thus $n_k = k^2/2 + \omega(k^{3/2})$ suffices for $F(k,n_k) = (1-o(1))k!$.

\subsection{Open Problems}

To recap Sections~\ref{loworder} and \ref{otherkdfa}, we find the following problems concerning lower-order terms interesting.
\begin{prob}\label{p2} Is there $c_1<7/4$ such that
\[k^2-O(k^{c_1}) \le f(k;k)?\]It is known that $c_1$ must be taken to be $\ge 1$.
\end{prob}
\begin{prob}\label{p3} Is there $c_2<7/4$ such that
\[\frac{k^2+k}{2}-O(k^{c_2}) \le f(k;k+1)?\]
It is possible that no error term is needed, and $(k^2+k)/2 = f(k;k+1)$ simply holds.
\end{prob}\label{p4}
\begin{prob} Is there $c_3<7/4$ such that
\[ G\left(k,\frac{k^2}{2}-\Omega(k^{c_3}),k^2\right) = o(k!) ?\]Due to Section~\ref{dstrat}, it is clear that $c_3$ must be taken so that $c_3> 3/2$ (but potentially we can take $c_3$ to be any value $>3/2$).
\end{prob}

It would also be interesting to extend the conclusion of Corollary~\ref{asy} to alphabets with linearly many extra letters. Specifically, we pose the following problem.
\begin{prob}\label{linear} Does there exist $\delta > 0$ such that $f(k;(1+\delta)k) \ge (1/2-o(1))k^2$? 
\end{prob}\noindent This would require a significant new idea. In particular, we think a proof would use some ``redundancy result'' to replace Remark~\ref{nfromk}. 

We further remark that the stronger statement, which claims $f(k;Ck)\ge (1/2-o(1))k^2$ for every $C>1$, could quite possibly be true. However, our methods fail to prove that $f(k;1.0001k)\ge (1/4-o(1))k^2$, so this currently seems out of reach. While we believe Problems~1-5 have affirmative answers, we are uncertain whether this stronger statement holds true. Our (lack of) understanding about more efficient superpatterns on small alphabets will be further discussed in \cite{hunter}.

\section{Acknowledgements}
The author would like to thank Daniel Carter and Zachary Chase for helpful conversations and looking at previous drafts of this paper. The author would also like to thank Carla Groenland for many suggestions on the presentation of the paper. Lastly, the author thanks Vincent Vatter and Mihir Singhal for giving comments on the final draft of this preprint.

\end{document}